\documentclass[a4paper,12pt]{amsart}
\usepackage[a4paper, total={6.5in, 10.5in}]{geometry}
\usepackage{amsmath, amsthm, amssymb, amsopn, amscd, mathtools, array, enumerate, amsfonts, graphics, hyperref, wasysym, chngcntr, linguex, graphicx, ytableau,youngtab,young,tikz}

\newtheorem{theorem}{Theorem}[section]
\newtheorem{corollary}[theorem]{Corollary}
\newtheorem{lemma}[theorem]{Lemma}

\newtheorem{lemma and definition}[theorem]{Lemma and Definition}
\newtheorem{proposition}[theorem]{Proposition}
\newtheorem{definition}[theorem]{Definition}
\newtheorem{exam}[theorem]{Example}

\newtheorem{remark}[theorem]{Remark}

\newtheorem{the construction}[theorem]{THE CONSTRUCTION}

\newtheorem{conjecture}[theorem]{Conjecture}

\newcommand{\field}[1]{\mathbb{#1}}

\newcommand{\Z }{\field{Z}}
\newcommand{\N }{\field{N}}

\DeclareMathOperator{\supp}{Supp}

\DeclareMathOperator{\kerr}{Ker}

\DeclareMathOperator{\rf}{RF}
\DeclareMathOperator{\ap}{Ap}
\DeclareMathOperator{\pf}{PF}
\DeclareMathOperator{\ed}{e}
\DeclareMathOperator{\m}{m}
\DeclareMathOperator{\maxx}{Maximals}
\DeclareMathOperator{\co}{C}
\DeclareMathOperator{\fr}{F}
\DeclareMathOperator{\ty}{t}

\begin{document}
	
\title[]{Row-Factorization Matrices in Arf Numerical Semigroups and Defining Ideals}
\author{Meral SÜER}
\address{Department of Mathematics, Faculty of Science and Letters, Batman University, Batman, Turkey}
\email{meral.suer@batman.edu.tr}
\author{Mehmet YEŞİL}
\address{Department of Mathematics, Faculty of Science and Letters, Batman University, Batman, Turkey}
\email{mehmet-yesil@outlook.com}
\maketitle

\begin{abstract}
In this paper, we investigate the row-factorization matrices of Arf numerical semigroups, and we provide the full list of such matrices of certain Arf numerical semigroups. We use the information of row-factorization matrices to detect the generic nature and to find generators of the defining ideals.
\end{abstract}

\maketitle

\section{Introduction}
Numerical semigroups and numerical semigroup rings are significant concepts that have many applications in ring theory, algebraic geometry and combinatorics. To understand the structure of a numerical semigroup, one must consider various notions, including the Frobenius number, pseudo-Frobenius numbers, multiplicity, embedding dimension, type and so on. In \cite{am}, A. Moscariello has introduced a new concept called  row-factorization matrices of pseudo-Frobenius numbers to investigate the type of almost symmetric numerical semigroups of embedding dimension four. Row-factorization matrices can be used for exploring many significant properties of the defining ideals. For instance, they can be used for detecting the generic nature of defining ideals of numerical semigroup rings. The connection between numerical semigroup rings and row-factorization matrices is further studied by J. Herzog, K. Watanabe in \cite{HW} and independently by K. Eto in \cite{eto3}, \cite{eto2} and \cite{eto1}.

Throughout this paper, $\Z,\N$ and $\N_{0}$ will denote the set of integers, positive integers and non-negative integers, respectively. A subset $S \subseteq \N_{0}$ containing $0$ is called a numerical semigroup if it is closed under addition and has finite complement in $\N_{0}$. Let $S$ be a numerical semigroup minimally generated by the set $\{ n_{1},n_{2},\dots,n_{e} \}$, $\Bbbk$ be a field and $t$ be an indeterminate. The subring $\Bbbk[S]=\Bbbk[t^{s}\mid s\in S]$ of the polynomial ring $\Bbbk[t]$ is called the semigroup ring of $S$ over $\Bbbk$. There is a surjective ring homomorphism from the polynomial ring $R=\Bbbk[x_{1},\dots,x_{e}]$ to $\Bbbk[S]$ mapping $x_{i}$ to $t^{n_{i}}$ whose kernel is denoted by $I_{S}$. We can represent $\Bbbk[S]$ as a quotient ring of $R$ with $R/I_{S}$, and we call $I_{S}$ the defining ideal of $\Bbbk[S]$. The ideal $I_{S}$ is a binomial prime ideal whose binomials correspond to the pairs of factorizations of elements in $S$ into the minimal generators. Binomial prime ideals are called toric ideals and the ring $\Bbbk[S]$ is called a toric ring.

A toric ideal is called generic if it has a minimal generating set consisting of binomials with full support. The genericity for toric ideals is introduced by I. Peeva and B. Sturmfels in \cite{PS} and they give the minimal free resolution of a toric ring when the toric ideal is generic. In \cite{eto1}, K. Eto gives an explicit relation between the genericity of a toric ideal and row-factorization matrices. He proves that the defining ideal of $\Bbbk[S]$ is not generic if $S$ is an almost symmetric numerical semigroup with embedding dimension bigger than three. Recently, in \cite{OM}, O.P. Bhardwaj, K. Goel and I. Sengupta give descriptions of possible row-factorization matrices of $S$ when it is minimally generated by an almost arithmetic sequence. They use these descriptions to find generators of the defining ideal when $S$ is also symmetric and has embedding dimension four or five.

In this paper, we consider Arf numerical semigroups with certain multiplicities. Those numerical semigroups are explicitly parametrized in \cite{ghkr} by P.A. Garc\'ia-S\'anches, B.A. Heredia, H.\.I. Karaka\c{s} and J.C. Rosales. We use the characterizations of Arf numerical semigroups from \cite{ghkr} to find pseudo-Frobenius numbers and we give the full list of row-factorization matrices of Arf numerical semigroups with multiplicity up to five. We also find the row-factorization matrices of certain Arf numerical semigroups whose conductor is a multiple of the multiplicity. Furthermore, we characterize the generic nature of defining ideals of the ring $\Bbbk[S]$ when $S$ is an Arf numerical semigroup and we find the generators of those ideals in some cases.

This paper is organised as follows. In Section \ref{sec2}, we gather necessary background of numerical semigroups, Arf numerical semigroups and row-factorization matrices that we use in latter sections. In Section \ref{sec3}, we provide a full list of row-factorization matrices of Arf numerical semigroups with multiplicities up to five. We conclude that the Frobenius numbers of these numerical semigroups have at least one row-factorization matrix, and the absolute value of its determinant is the Frobenius number itself. In Section \ref{sec4}, we give parametrizations of row-factorization matrices of certain Arf numerical semigroups whose conductor is a multiple of the multiplicity (Proposition \ref{p10}). We show that the Frobenius number of those semigroups has a row-factorization matrix whose determinant's absolute value is the Frobenius number, Corollary \ref{cor2}. We also show in Lemma \ref{lm} that for an Arf numerical semigroup with multiplicity bigger than five, the row-factorization matrices of the Frobenius number have at least one column with similar entries. In Section \ref{sec5}, we use the information that we gathered in previous sections to detect the generic nature and find generators of defining ideals of numerical semigroup rings obtained by Arf numerical semigroups. In particular, we prove the main results Theorems \ref{thm1}, \ref{thm2}, \ref{thm3}, \ref{thm4} of this paper and interpret the results with Arf rings.

\section{Preliminaries}\label{sec2}

Let $S$ be a numerical semigroup. The smallest non-zero element in $S$ is called the multiplicity of $S$ and denoted by $\m(S)$. The elements in $\N\setminus S$ are called the gaps of $S$, and the greatest gap is called the Frobenius number of $S$ and denoted by $\fr(S)$. In particular, if $S=\N$, $\fr(S)=-1$. The number $\fr(S)+1$ is called the conductor of $S$ and denoted by $\co(S)$.

Let $A\subset S$. We say that $S$ is generated by $A$ and write $S=\langle A \rangle$ whenever $S=\{ \sum_{i=1}^{e}\lambda_{i}n_{i} \mid n_{i}\in A,\lambda_{i}\in\N_{0},e\in \N \}$. It is well known that every numerical semigroup is finitely generated, i.e. the set $A$ is always finite. If $S$ is generated by $A$ and there is no proper subset of $A$ that generates $S$, then we say that $S$ is minimally generated by $A$, and in this case the cardinality of $A$ is called the embedding dimension of $S$, and denoted by $\ed(S)$.

Let $n\in S$ be a non-zero element. We define the Ap\'ery set of $S$ with respect to $n$, denoted by $\ap(S,n)$, as follows $$\ap(S,n)=\{ s\in S\mid s-n \notin S \}.$$ It is well known that $$\ap(S,n)=\{ \omega(0)=0,\omega(1),\dots,\omega(n-1) \}$$ where $\omega(i)=\min\{ s\in S \mid s \mod n = i \}$ for each $i\in\{ 1,\dots,n-1\}$.

Let $z\in\Z$. We say that $z$ is a pseudo-Frobenius number of $S$ if $z\notin S$ and $z+s\in S$ for all non-zero element $s\in S$. The set of pseudo-Frobenius numbers of $S$ is denoted by $\pf(S)$, and the cardinality of $\pf(S)$ is called the type of $S$, denoted by $\ty(S)$. For $z,w\in\Z$, we say that $z\leq_{S}w$ if $w-z\in S$, which defines a partially ordered relation. For a subset $A$ of $S$, $\maxx_{\leq_S}(A)$ will denote the maximal elements of $A$ with respect to the relation $\leq_{S}$.  Next proposition gives a characterization of $\pf(S)$ in terms of the Ap\'ery sets of $S$.

\begin{proposition} \label{1}\cite[Proposition 2.20]{RG}
	Let $S$ be numerical semigroup and $n\neq0$ an element of $S$. Then \[\pf(S)=\{ w-n \mid w\in \maxx_{\leq S}(\ap(S,n)) \}.\]
\end{proposition}

It is also well known that the embedding dimension of $S$ is less than or equal to the multiplicity of $S$, i.e. $\ed(S) \leq \m(S)$. When $\ed(S)=\m(S)$, we say that $S$ has maximal embedding dimension, in short $S$ is a MED-numerical semigroup. Next proposition gives an exact characterization of MED property in terms of the minimal generators and 
the Ap\'ery sets.

\begin{proposition} \label{2}\cite[Proposition 3.1]{RG}
	Let $S$ be numerical semigroup minimally generated by $\{ n_{1}<n_{2}<\dots<n_{e}\}$. Then $S$ has maximal embedding dimension if and only if $\ap(S,n_{1})=\{ 0,n_{2},\dots,n_{e} \}$.
\end{proposition}

A numerical semigroup $S$ is Arf if for every $s_{1}, s_{2}, s_{3}\in S$ with $s_{1}\leq s_{2}\leq s_{3}$ we have the property that $s_{2}+s_{3}-s_{1}\in S$. This is equivalent to prove that for every $s_{1}, s_{2}\in S$ with $s_{1}\leq s_{2}$ we have $2s_{2}-s_{1}\in S$. It can easily be seen from the Arf property that Arf numerical semigroups have maximal embedding dimension but the inverse statement is not true always. By Propositions \ref{1} and \ref{2}, it is immediate to see that for an Arf numerical semigroup $S$ with multiplicity $\m(S)=m$, the set $(\ap(S,m)\cup\{m\})\setminus\{0\}=\{ m,\omega(1),\dots,\omega(m-1) \}$ forms the minimal generating system for $S$. Moreover, if $S$ is minimally generated by $\{ n_{1}<n_{2}<\dots<n_{e}\}$, then $$\pf(S)=\{ n_{2}-n_{1},n_{3}-n_{1},\dots,n_{e}-n_{1} \}$$ and $\ty(S)=\m(S)-1$. That is, the number of pseudo-Frobenius elements of an Arf numerical semigroup is one less than the multiplicity. Next proposition gives explicit parametrizations of Arf numerical semigroups with a given multiplicity up to five and conductor. 

\begin{proposition} \cite[Section 7]{ghkr} \label{arf}
	Let $S$ be an Arf numerical semigroup with multiplicity $\m(S)=m$ and conductor $\co(S)=s$. Then
	\begin{enumerate}
		\item if $m=1$ then $S=\N$,
		\item if $m=2$ then $s$ is a positive even integer and $S=\langle 2,s+1 \rangle$,
		\item if $m=3$, $s\geq 3$ and $s \equiv 0\mod 3$ then $S=\langle 3,s+1,s+2\rangle$,
		\item if $m=3$, $s> 3$ and $s \equiv 2\mod 3$ then $S=\langle 3,s,s+2\rangle$,
		\item if $m=4$, $s\geq 4$ and $s \equiv 0\mod 4$ then $S=\langle 4,4k+2,s+1,s+3\rangle$ for some $k\in \{ 1,\dots, \frac{s}{4} \}$,
		\item if $m=4$, $s> 4$ and $s \equiv 2\mod 4$ then $S=\langle 4,4k+2,s+1,s+3\rangle$ for some $k\in \{ 1,\dots, \frac{s-2}{4} \}$,
		\item if $m=4$, $s> 4$ and $s \equiv 3\mod 4$ then $S=\langle 4,s,s+2,s+3\rangle$,
		\item if $m=5$, $s\geq 5$ and $s \equiv 0\mod 5$ then either $S=\langle 5,s-2,s+1,s+2,s+4\rangle$ or $S=\langle 5,s+1,s+2,s+3,s+4\rangle$,
		\item if $m=5$, $s>5$ and $s \equiv 2\mod 5$ then $S=\langle 5,s,s+1,s+2,s+4\rangle$,
		\item if $m=5$, $s> 5$ and $s \equiv 3\mod 5$ then $S=\langle 5,s,s+1,s+3,s+4\rangle$,
		\item if $m=5$, $s> 5$ and $s \equiv 4\mod 5$ then either $S=\langle 5,s-2,s,s+2,s+4\rangle$ or $S=\langle 5,s,s+2,s+3,s+4\rangle$.
	\end{enumerate}
\end{proposition}

We now recall the notion of a row-factorization matrix, RF-matrix in short, of a numerical semigroup $S$ minimally generated by $\{ n_{1}<n_{2}<\dots<n_{e} \}$.

\begin{definition}
	Let $f\in\pf(S)$. An $e\times e$ matrix $[a_{ij}]$ is an RF-matrix of $f$ if $a_{ii}=-1$, $a_{ij}\in\N_{0}$ for $i\neq j$ and $f=\sum_{j=1}^{e}a_{ij}n_{j}$ for each $i\in\{ 1,2,\dots,e\}$.
\end{definition}

\section{Row-Factorization Matrices of Arf Numerical Semigroups with Multiplicity up to 5}\label{sec3}

In this section, we provide complete characterizations of the row-factorization matrices of Arf numerical semigroups with multiplicity smaller than six. We will use Propositions \ref{2} and \ref{arf} repeatedly without referring. 

\begin{proposition} \label{p1}
	Let $S$ be an Arf numerical semigroup with multiplicity $2$ and a positive even conductor $s$. Then $S=\langle 2,s+1 \rangle$, $\pf(S)=\{ s-1 \}$ and \[ \rf(s-1)= \begin{bmatrix}
		-1&1\\
		s&-1 
	\end{bmatrix}. \]
\end{proposition}

\begin{proof}
	If $S$ is an Arf numerical semigroup with multiplicity $2$ whose conductor $s$ is a positive even integer, we know that $S= \langle 2,s+1 \rangle$. Therefore, $\ap(S,2)=\{ 0,s+1 \}$, and so $\pf(S)=\{ s-1 \}$. In order to find an RF-matrix of $s-1$, one needs to solve by definition the following equations for non-negative integers $a_{12}$ and $a_{21}$ which are trivially found to be $1$ and $s$, respectively
	\[\begin{array}{rcrll}
		s-1&=&-1\times 2&+&a_{12}\times(s+1),\\
		s-1&=&a_{21}\times 2&-&1\times (s+1).
	\end{array}\]
	Hence, $\rf(s-1)$ is found as desired.
\end{proof}

\begin{proposition}\label{p2}
	Let $S$ be an Arf numerical semigroup with multiplicity $3$ and conductor $s\geq 3$. If $s \equiv 0\mod 3$, then $S=\langle 3,s+1,s+2\rangle$, $\pf(S)=\{ s-2,s-1 \}$ and
	\[ \rf(s-2)= \begin{bmatrix}
		-1&1&0\\
		\frac{s-3}{3}&-1&1\\
		\frac{2s}{3}&0&-1 
	\end{bmatrix}, \rf(s-1)= \begin{bmatrix}
		-1&0&1\\
		\frac{2s}{3}&-1&0\\
		\frac{s}{3}&1&-1 
	\end{bmatrix}. \]
\end{proposition}

\begin{proof}
	If $s \equiv 0 \mod 3$, we know that $S=\langle 3,s+1,s+2\rangle$. Therefore, $\ap(S,3)=\{ 0,s+1,s+2 \}$, and so $\pf(S)=\{ s-2,s-1 \}$. To find row-factorization matrices of $s-2$ and $s-1$, we need to find the coefficients of following factorizations
	\[\begin{array}{rlrlrlr}
		s-2&=&-1\times3&+&a_{12}\times(s+1)&+&a_{13}\times(s+2),\\
		s-2&=&a_{21}\times3&-&1\times(s+1)&+&a_{23}\times(s+2),\\
		s-2&=&a_{31}\times3&+&a_{32}\times(s+1)&-&1\times(s+2),
	\end{array}\]
	and
	\[\begin{array}{rlrlrlr}
		s-1&=&-1\times3&+&b_{12}\times(s+1)&+&b_{13}\times(s+2),\\
		s-1&=&b_{21}\times3&-&1\times(s+1)&+&b_{23}\times(s+2),\\
		s-1&=&b_{31}\times3&+&b_{32}\times(s+1)&-&1\times(s+2).
	\end{array}\]
	By definition of row-factorization matrices, and since $s \equiv 0\mod 3$, it is easy to find that $a_{12}=1$, $a_{13}=0$, $a_{21}=\frac{s-3}{3}$, $a_{23}=1$, $a_{31}=\frac{2s}{3}$, $a_{32}=0$, and $b_{12}=0$, $b_{13}=1$, $b_{21}=\frac{2s}{3}$, $b_{23}=0$,  $b_{31}=\frac{s}{3}$, $b_{32}=1$. Thus, $\rf(s-2)$ and $\rf(s-1)$ have been found as desired.
\end{proof}

Next proposition can be proved in a similar way to the proof of Proposition \ref{p2}.

\begin{proposition} \label{p21}
	Let $S$ be an Arf numerical semigroup with multiplicity $3$ and conductor $s>3$. If $s \equiv 2\mod 3$, then $S=\langle 3,s,s+2\rangle$, $\pf(S)=\{ s-3,s-1 \}$ and
	\[ \rf(s-3)= \begin{bmatrix}
		-1&1&0\\
		\frac{s-5}{3}&-1&1\\
		\frac{2s-1}{3}&0&-1 
	\end{bmatrix}, \rf(s-1)= \begin{bmatrix}
		-1&0&1\\
		\frac{2s-1}{3}&-1&0\\
		\frac{s+1}{3}&1&-1 
	\end{bmatrix}. \]
\end{proposition}

Next we provide the characterizations of row-factorization matrices of Arf numerical semigroups with multiplicity $4$. We will give a complete list of these matrices.we will leave the proofs since they only involve finding factorizations for the pseudo-Frobenius numbers. Also, one needs to be careful, because for a pseudo-Frobenius number there might be different factorizations, and so it may have more than one row-factorization matrix.

\begin{proposition}\label{p3}
	Let $S$ be an Arf numerical semigroup with multiplicity $4$ and conductor $s\geq 4$. If $s \equiv 0\mod 4$, then either
	\begin{enumerate}
		\item $S=\langle 4,4k+2,s+1,s+3\rangle$ for some $1\leq k<\frac{s}{4}$, $\pf(S)=\{ 4k-2,s-3,s-1 \}$ and
		\[ \rf(4k-2)= \begin{bmatrix}
			-1&1&0&0\\
			2k&-1&0&0\\
			k-1&0&-1&1\\
			k&0&1&-1 
		\end{bmatrix},\]
		\[\rf(s-3)= \begin{bmatrix}
			-1&0&1&0\\
			k-1&-1&0&1\\
			\frac{s}{2}-a-(2a-1)k &2a-1&-1&0\\
			\frac{s}{2}-b-2bk &2b&0&-1 \end{bmatrix}, \]
		\[\text{ where } 1\leq a \leq \frac{s+2k}{2(2k+1)}, 0\leq b \leq \frac{s}{2(2k+1)}, \text{ and } a,b\in \N_0, \]
		
		\[ \rf(s-1)= \begin{bmatrix}
			-1&0&0&1\\
			k&-1&1&0\\
			\frac{s}{2}-b-2bk&2b&-1&0\\
			\frac{s}{2}-a+1-(2a-1)k&2a-1&0&-1 \end{bmatrix} \text{or}
		\begin{bmatrix}
			-1&0&0&1\\
			k&-1&1&0\\
			\frac{s}{2}-b-2bk&2b&-1&0\\
			0&0&2&-1 \end{bmatrix},
		\]
		\[\text{ where } 1\leq a \leq \frac{s+2k+2}{2(2k+1)}, 0\leq b \leq \frac{s}{2(2k+1)}, \text{ and } a,b\in \N_0, \]
		or
		\item $S=\langle 4,s+1,s+2,s+3\rangle$, $\pf(S)=\{s-3,s-2,s-1 \}$ and
		\[ \rf(s-3)= \begin{bmatrix}
			-1&1&0&0\\
			\frac{s-4}{4}&-1&1&0\\
			\frac{s-4}{4}&0&-1&1\\
			\frac{s}{2}&0&0&-1 
		\end{bmatrix}, \rf(s-2)= \begin{bmatrix}
			-1&0&1&0\\
			\frac{s-4}{4}&-1&0&1\\
			\frac{s}{2}&0&-1&0\\
			\frac{s}{4}&1&0&-1 \end{bmatrix}, \]
		\[ \rf(s-1)= \begin{bmatrix}
			-1&0&0&1\\
			\frac{s}{2}&-1&0&0\\
			\frac{s}{4}&1&-1&0\\
			\frac{s}{4}&0&1&-1 \end{bmatrix} \text{or}
		\begin{bmatrix}
			-1&0&0&1\\
			\frac{s}{2}&-1&0&0\\
			\frac{s}{4}&1&-1&0\\
			0&2&0&-1 \end{bmatrix}.
		\]
	\end{enumerate}
\end{proposition}

\begin{proposition} \label{p31}
	Let $S$ be an Arf numerical semigroup with multiplicity $4$ and conductor $s>4$.
	\begin{enumerate}
		\item If $s \equiv 2\mod 4$, then $S=\langle 4,4k+2,s+1,s+3\rangle$ for some $k \in \{ 1,\dots,\frac{s-2}{4} \}$, $\pf(S)=\{ 4k-2,s-3,s-1 \}$ and
		\[ \rf(4k-2)= \begin{bmatrix}
			-1&1&0&0\\
			2k&-1&0&0\\
			k-1&0&-1&1\\
			k&0&1&-1 
		\end{bmatrix},\]
		\[\rf(s-3)= \begin{bmatrix}
			-1&0&1&0\\
			k-1&-1&0&1\\
			\frac{s}{2}-a-(2a-1)k&2a-1&-1&0\\
			\frac{s}{2}-b-2bk&2b&0&-1 \end{bmatrix}, \]
		\[\text{ where } 1\leq a \leq \frac{s+2k}{2(2k+1)}, \text{ and } 0\leq b \leq \frac{s}{2(2k+1)}, a,b\in \N_0, \]
		
		\[ \rf(s-1)= \begin{bmatrix}
			-1&0&0&1\\
			k&-1&1&0\\
			\frac{s}{2}-b-bk&2b&-1&0\\
			\frac{s}{2}-a+1-(2a-1)k&2a-1&0&-1 \end{bmatrix} \text{or}
		\begin{bmatrix}
			-1&0&0&1\\
			k&-1&1&0\\
			\frac{s}{2}-b-2bk&2b&-1&0\\
			0&0&2&-1 \end{bmatrix},
		\]
		\[\text{ where } 1\leq a \leq \frac{s+2k+2}{2(2k+1)}, \text{ and } 0\leq b \leq \frac{s}{2(2k+1)}, a,b\in \N_0, \]
		\item if $s \equiv 3\mod 4$, then $S=\langle 4,s,s+2,s+3\rangle$, $\pf(S)=\{ s-4,s-2,s-1 \}$ and
		\[ \rf(s-4)= \begin{bmatrix}
			-1&1&0&0\\
			\frac{s-7}{4}&-1&0&1\\
			\frac{s-1}{2}&0&-1&0\\
			\frac{s-3}{4}&0&1&-1 
		\end{bmatrix}, \] 
		\[ \rf(s-2)= \begin{bmatrix}
			-1&0&1&0\\
			\frac{s-1}{2}&-1&0&0\\
			\frac{s-3}{4}&0&-1&1\\
			\frac{s+1}{4}&1&0&-1 \end{bmatrix} \text{or}
		\begin{bmatrix}
			-1&0&1&0\\
			\frac{s-1}{2}&-1&0&0\\
			0&2&-1&0\\
			\frac{s+1}{4}&1&0&-1 \end{bmatrix},
		\]
		\[ \rf(s-1)= \begin{bmatrix}
			-1&0&0&1\\
			\frac{s-3}{4}&-1&1&0\\
			\frac{s+1}{4}&1&-1&0\\
			0&1&1&-1 \end{bmatrix} \text{or} \begin{bmatrix}
			-1&0&0&1\\
			\frac{s-3}{4}&-1&1&0\\
			\frac{s+1}{4}&1&-1&0\\
			\frac{s+1}{2}&0&0&-1 \end{bmatrix}.  \]
	\end{enumerate}
\end{proposition}


Now we provide the characterizations of row-factorization matrices of Arf numerical semigroups with multiplicity $5$. By Proposition \ref{arf}, we know that these numerical semigroups formulated in six different ways. We will give a complete list of these matrices. We will leave the proofs again since they can be proved by finding the factorizations of pseudo-Frobenius numbers.

\begin{proposition} \label{p4}
	Let $S$ be an Arf numerical semigroup with multiplicity $5$ and conductor $s>5$. If $s \equiv 0\mod 5$ and $S=\langle 5,s-2,s+1,s+2,s+4\rangle$, then $\pf(S)=\{ s-7,s-4,s-3,s-1 \}$ and
	\[ \rf(s-7)= \begin{bmatrix}
		-1&1&0&0&0\\
		\frac{s-10}{5}&-1&1&0&0\\
		\frac{s-10}{5}&0&-1&0&1\\
		\frac{2s-5}{5}&0&0&-1&0\\
		\frac{s-5}{5}&0&0&1&-1 
	\end{bmatrix},\]
	\[\rf(s-4)= \begin{bmatrix}
		-1&0&1&0&0\\
		\frac{s-10}{5}&-1&0&0&1\\
		\frac{s-5}{5}&0&-1&1&0\\
		\frac{s}{5}&1&0&-1&0\\
		\frac{2s}{5}&0&0&0&-1 
	\end{bmatrix} \text{or} \begin{bmatrix}
		-1&0&1&0&0\\
		\frac{s-10}{5}&-1&0&0&1\\
		\frac{s-5}{5}&0&-1&1&0\\
		\frac{s}{5}&1&0&-1&0\\
		0&1&0&1&-1 
	\end{bmatrix}, \]
	\[ \rf(s-3)= \begin{bmatrix}
		-1&0&0&1&0\\
		\frac{2s-5}{5}&-1&0&0&0\\
		\frac{s}{5}&1&-1&0&0\\
		\frac{s-5}{5}&0&0&-1&1\\
		\frac{s}{5}&0&1&0&-1 
	\end{bmatrix} \text{or} \begin{bmatrix}
		-1&0&0&1&0\\
		\frac{2s-5}{5}&-1&0&0&0\\
		\frac{s}{5}&1&-1&0&0\\
		0&1&1&-1&0\\
		\frac{s}{5}&0&1&0&-1 
	\end{bmatrix} \]
	\[ \text{or} \begin{bmatrix}
		-1&0&0&1&0\\
		\frac{2s-5}{5}&-1&0&0&0\\
		\frac{s}{5}&1&-1&0&0\\
		\frac{s-5}{5}&0&0&-1&1\\
		1&2&0&0&-1 
	\end{bmatrix} \text{or} \begin{bmatrix}
		-1&0&0&1&0\\
		\frac{2s-5}{5}&-1&0&0&0\\
		\frac{s}{5}&1&-1&0&0\\
		0&1&1&-1&0\\
		1&2&0&0&-1 
	\end{bmatrix}, \]
	\[ \rf(s-1)= \begin{bmatrix}
		-1&0&0&0&1\\
		\frac{s-5}{5}&-1&0&1&0\\
		\frac{2s}{5}&0&-1&0&0\\
		\frac{s}{5}&0&1&-1&0\\
		\frac{s+5}{5}&1&0&0&-1 
	\end{bmatrix} \text{or} \begin{bmatrix}
		-1&0&0&0&1\\
		\frac{s-5}{5}&-1&0&1&0\\
		0&1&-1&1&0\\
		\frac{s}{5}&0&1&-1&0\\
		\frac{s+5}{5}&1&0&0&-1 
	\end{bmatrix} \]
	\[
	\text{or} \begin{bmatrix}
		-1&0&0&0&1\\
		\frac{s-5}{5}&-1&0&1&0\\
		\frac{2s}{5}&0&-1&0&0\\
		\frac{s}{5}&0&1&-1&0\\
		0&0&1&1&-1 
	\end{bmatrix} \text{or}
	\begin{bmatrix}
		-1&0&0&0&1\\
		\frac{s-5}{5}&-1&0&1&0\\
		0&1&-1&1&0\\
		\frac{s}{5}&0&1&-1&0\\
		0&0&1&1&-1 
	\end{bmatrix} \]
	\[ \text{or} \begin{bmatrix}
		-1&0&0&0&1\\
		\frac{s-5}{5}&-1&0&1&0\\
		\frac{2s}{5}&0&-1&0&0\\
		1&2&0&-1&0\\
		\frac{s+5}{5}&1&0&0&-1 
	\end{bmatrix} \text{or} \begin{bmatrix}
		-1&0&0&0&1\\
		\frac{s-5}{5}&-1&0&1&0\\
		0&1&-1&1&0\\
		1&2&0&-1&0\\
		\frac{s+5}{5}&1&0&0&-1 
	\end{bmatrix} \]
	\[
	\text{or} \begin{bmatrix}
		-1&0&0&0&1\\
		\frac{s-5}{5}&-1&0&1&0\\
		\frac{2s}{5}&0&-1&0&0\\
		1&2&0&-1&0\\
		0&0&1&1&-1 
	\end{bmatrix} \text{or}
	\begin{bmatrix}
		-1&0&0&0&1\\
		\frac{s-5}{5}&-1&0&1&0\\
		0&1&-1&1&0\\
		1&2&0&-1&0\\
		0&0&1&1&-1 
	\end{bmatrix}. \]
\end{proposition}

		\begin{proposition} \label{51}
			Let $S$ be an Arf numerical semigroup with multiplicity $5$ and conductor $s\geq 5$. If $s \equiv 0\mod 5$ and $S=\langle 5,s+1,s+2,s+3,s+4\rangle$, then $\pf(S)=\{ s-4,s-3,s-2,s-1 \}$ and
			\[ \rf(s-4)= \begin{bmatrix}
				-1&1&0&0&0\\
				\frac{s-5}{5}&-1&1&0&0\\
				\frac{s-5}{5}&0&-1&1&0\\
				\frac{s-5}{5}&0&0&-1&1\\
				\frac{2s}{5}&0&0&0&-1 
			\end{bmatrix}, \rf(s-3)= \begin{bmatrix}
				-1&0&1&0&0\\
				\frac{s-5}{5}&-1&0&1&0\\
				\frac{s-5}{5}&0&-1&0&1\\
				\frac{2s}{5}&0&0&-1&0\\
				\frac{s}{5}&1&0&0&-1 
			\end{bmatrix}, \]
			\[ \rf(s-2)= \begin{bmatrix}
				-1&0&0&1&0\\
				\frac{s-5}{5}&-1&0&0&1\\
				\frac{2s}{5}&0&-1&0&0\\
				\frac{s}{5}&1&0&-1&0\\
				\frac{s}{5}&0&1&0&-1 
			\end{bmatrix} \text{or} \begin{bmatrix}
				-1&0&0&1&0\\
				\frac{s-5}{5}&-1&0&0&1\\
				\frac{2s}{5}&0&-1&0&0\\
				\frac{s}{5}&1&0&-1&0\\
				0&2&0&0&-1 
			\end{bmatrix},\]
			\[ \rf(s-1)= \begin{bmatrix}
				-1&0&0&0&1\\
				\frac{2s}{5}&-1&0&0&0\\
				\frac{s}{5}&1&-1&0&0\\
				\frac{s}{5}&0&1&-1&0\\
				\frac{s}{5}&0&0&1&-1 
			\end{bmatrix} \text{or} \begin{bmatrix}
				-1&0&0&0&1\\
				\frac{2s}{5}&-1&0&0&0\\
				\frac{s}{5}&1&-1&0&0\\
				0&2&0&-1&0\\
				\frac{s}{5}&0&0&1&-1 
			\end{bmatrix}\]
			\[ \text{or}
			\begin{bmatrix}
				-1&0&0&0&1\\
				\frac{2s}{5}&-1&0&0&0\\
				\frac{s}{5}&1&-1&0&0\\
				\frac{s}{5}&0&1&-1&0\\
				0&1&1&0&-1 
			\end{bmatrix} \text{or} \begin{bmatrix}
				-1&0&0&0&1\\
				\frac{2s}{5}&-1&0&0&0\\
				\frac{s}{5}&1&-1&0&0\\
				0&2&0&-1&0\\
				0&1&1&0&-1 
			\end{bmatrix}.\]
		\end{proposition}

		\begin{proposition} \label{52}
			Let $S$ be an Arf numerical semigroup with multiplicity $5$ and conductor $s>5$. If $s \equiv 2\mod 5$, then $S=\langle 5,s,s+1,s+2,s+4\rangle$, $\pf(S)=\{ s-5,s-4,s-3,s-1 \}$ and
			\[ \rf(s-5)= \begin{bmatrix}
				-1&1&0&0&0\\
				\frac{s-7}{5}&-1&0&1&0\\
				\frac{2s-4}{5}&0&-1&0&0\\
				\frac{s-7}{5}&0&0&-1&1\\
				\frac{s-2}{5}&0&1&0&-1 
			\end{bmatrix}, \] 
			\[ \rf(s-4)= \begin{bmatrix}
				-1&0&1&0&0\\
				\frac{2s-4}{5}&-1&0&0&0\\
				\frac{s-7}{5}&0&-1&0&1\\
				\frac{s-2}{5}&1&0&-1&0\\
				\frac{s-2}{5}&0&0&1&-1 
			\end{bmatrix} \text{or} \begin{bmatrix}
				-1&0&1&0&0\\
				\frac{2s-4}{5}&-1&0&0&0\\
				\frac{s-7}{5}&0&-1&0&1\\
				\frac{s-2}{5}&1&0&-1&0\\
				0&2&0&0&-1 
			\end{bmatrix},\]
			\[ \rf(s-3)= \begin{bmatrix}
				-1&0&0&1&0\\
				\frac{s-7}{5}&-1&0&0&1\\
				\frac{s-2}{5}&1&-1&0&0\\
				\frac{s-2}{5}&0&1&-1&0\\
				\frac{2s+1}{5}&0&0&0&-1 
			\end{bmatrix} \text{or} \begin{bmatrix}
				-1&0&0&1&0\\
				\frac{s-7}{5}&-1&0&0&1\\
				\frac{s-2}{5}&1&-1&0&0\\
				\frac{s-2}{5}&0&1&-1&0\\
				0&1&1&0&-1 
			\end{bmatrix}, \]
			\[ \rf(s-1)= \begin{bmatrix}
				-1&0&0&0&1\\
				\frac{s-2}{5}&-1&1&0&0\\
				\frac{s-2}{5}&0&-1&1&0\\
				\frac{2s+1}{5}&0&0&-1&0\\
				\frac{s+3}{5}&1&0&0&-1 
			\end{bmatrix} \text{or} \begin{bmatrix}
				-1&0&0&0&1\\
				\frac{s-2}{5}&-1&1&0&0\\
				0&2&-1&0&0\\
				\frac{2s+1}{5}&0&0&-1&0\\
				\frac{s+3}{5}&1&0&0&-1 \end{bmatrix} \]
			\[
			\text{or} \begin{bmatrix}
				-1&0&0&0&1\\
				\frac{s-2}{5}&-1&1&0&0\\
				\frac{s-2}{5}&0&-1&1&0\\
				0&1&1&-1&0\\
				\frac{s+3}{5}&1&0&0&-1 
			\end{bmatrix} \text{or} \begin{bmatrix}
				-1&0&0&0&1\\
				\frac{s-2}{5}&-1&1&0&0\\
				\frac{s-2}{5}&0&-1&1&0\\
				\frac{2s+1}{5}&0&0&-1&0\\
				0&0&1&1&-1 
			\end{bmatrix}  
			\]
			\[ \text{or} \begin{bmatrix}
				-1&0&0&0&1\\
				\frac{s-2}{5}&-1&1&0&0\\
				0&2&-1&0&0\\
				0&1&1&-1&0\\
				\frac{s+3}{5}&1&0&0&-1 
			\end{bmatrix} \text{or} \begin{bmatrix}
				-1&0&0&0&1\\
				\frac{s-2}{5}&-1&1&0&0\\
				0&2&-1&0&0\\
				\frac{2s+1}{5}&0&0&-1&0\\
				0&0&1&1&-1 
			\end{bmatrix}\]
			\[
			\text{or}
			\begin{bmatrix}
				-1&0&0&0&1\\
				\frac{s-2}{5}&-1&1&0&0\\
				\frac{s-2}{5}&0&-1&1&0\\
				0&1&1&-1&0\\
				0&0&1&1&-1 
			\end{bmatrix} \text{or}
			\begin{bmatrix}
				-1&0&0&0&1\\
				\frac{s-2}{5}&-1&1&0&0\\
				0&2&-1&0&0\\
				0&1&1&-1&0\\
				0&0&1&1&-1 
			\end{bmatrix}.
			\]
		\end{proposition}

		\begin{proposition} \label{53}
			Let $S$ be an Arf numerical semigroup with multiplicity $5$ and conductor $s>5$. If $s \equiv 3\mod 5$, then $S=\langle 5,s,s+1,s+3,s+4\rangle$, $\pf(S)=\{ s-5,s-4,s-2,s-1 \}$ and
			\[ \rf(s-5)= \begin{bmatrix}
				-1&1&0&0&0\\
				\frac{s-8}{5}&-1&0&1&0\\
				\frac{s-8}{5}&0&-1&0&1\\
				\frac{s-3}{5}&0&1&-1&0\\
				\frac{2s-1}{5}&0&0&0&-1 
			\end{bmatrix}, \] 
			\[ \rf(s-4)= \begin{bmatrix}
				-1&0&1&0&0\\
				\frac{s-8}{5}&-1&0&0&1\\
				\frac{s-3}{5}&1&-1&0&0\\
				\frac{2s-1}{5}&0&0&-1&0\\
				\frac{s-3}{5}&0&0&1&-1 
			\end{bmatrix} \text{or} \begin{bmatrix}
				-1&0&1&0&0\\
				\frac{s-8}{5}&-1&0&0&1\\
				\frac{s-3}{5}&1&-1&0&0\\
				\frac{2s-1}{5}&0&0&-1&0\\
				0&2&0&0&-1 
			\end{bmatrix},\]
			\[ \rf(s-2)= \begin{bmatrix}
				-1&0&0&1&0\\
				\frac{s-3}{5}&-1&1&0&0\\
				\frac{2s-1}{5}&0&-1&0&0\\
				\frac{s-3}{5}&0&0&-1&1\\
				\frac{s+2}{5}&1&0&0&-1 
			\end{bmatrix} \text{or} \begin{bmatrix}
				-1&0&0&1&0\\
				\frac{s-3}{5}&-1&1&0&0\\
				\frac{2s-1}{5}&0&-1&0&0\\
				0&1&1&-1&0\\
				\frac{s+2}{5}&1&0&0&-1 
			\end{bmatrix} \]
			\[ \text{or} \begin{bmatrix}
				-1&0&0&1&0\\
				\frac{s-3}{5}&-1&1&0&0\\
				\frac{2s-1}{5}&0&-1&0&0\\
				\frac{s-3}{5}&0&0&-1&1\\
				0&0&2&0&-1 
			\end{bmatrix} \text{or} \begin{bmatrix}
				-1&0&0&1&0\\
				\frac{s-3}{5}&-1&1&0&0\\
				\frac{2s-1}{5}&0&-1&0&0\\
				0&1&1&-1&0\\
				0&0&2&0&-1 
			\end{bmatrix}, \]
			\[ \rf(s-1)= \begin{bmatrix}
				-1&0&0&0&1\\
				\frac{2s-1}{5}&-1&0&0&0\\
				\frac{s-3}{5}&0&-1&1&0\\
				\frac{s+2}{5}&1&0&-1&0\\
				\frac{s+2}{5}&0&1&0&-1 
			\end{bmatrix} \text{or} \begin{bmatrix}
				-1&0&0&0&1\\
				\frac{2s-1}{5}&-1&0&0&0\\
				\frac{s-3}{5}&0&-1&1&0\\
				\frac{s+2}{5}&1&0&-1&0\\
				0&1&0&1&-1 
			\end{bmatrix} \]
			\[
			\text{or} \begin{bmatrix}
				-1&0&0&0&1\\
				\frac{2s-1}{5}&-1&0&0&0\\
				\frac{s-3}{5}&0&-1&1&0\\
				0&0&2&-1&0\\
				\frac{s+2}{5}&0&1&0&-1 
			\end{bmatrix} \text{or}
			\begin{bmatrix}
				-1&0&0&0&1\\
				\frac{2s-1}{5}&-1&0&0&0\\
				0&2&-1&0&0\\
				\frac{s+2}{5}&1&0&-1&0\\
				\frac{s+2}{5}&0&1&0&-1 
			\end{bmatrix} \]
			\[
			\text{or} \begin{bmatrix}
				-1&0&0&0&1\\
				\frac{2s-1}{5}&-1&0&0&0\\
				0&2&-1&0&0\\
				0&0&2&-1&0\\
				\frac{s+2}{5}&0&1&0&-1 
			\end{bmatrix} \text{or}
			\begin{bmatrix}
				-1&0&0&0&1\\
				\frac{2s-1}{5}&-1&0&0&0\\
				0&2&-1&0&0\\
				\frac{s+2}{5}&1&0&-1&0\\
				0&1&0&1&-1 
			\end{bmatrix} \]
			\[
			\text{or} \begin{bmatrix}
				-1&0&0&0&1\\
				\frac{2s-1}{5}&-1&0&0&0\\
				\frac{s-3}{5}&0&-1&1&0\\
				0&0&2&-1&0\\
				0&1&0&1&-1 
			\end{bmatrix} \text{or}
			\begin{bmatrix}
				-1&0&0&0&1\\
				\frac{2s-1}{5}&-1&0&0&0\\
				0&2&-1&0&0\\
				0&0&2&-1&0\\
				0&1&0&1&-1 
			\end{bmatrix}. \]
		\end{proposition}
		
		\begin{proposition} \label{54}
			Let $S$ be an Arf numerical semigroup with multiplicity $5$ and conductor $s>5$. If $s \equiv 4\mod 5$ and $S=\langle 5,s-2,s,s+2,s+4\rangle$, then $\pf(S)=\{ s-7,s-5,s-3,s-1 \}$ and
			\[ \rf(s-7)= \begin{bmatrix}
				-1&1&0&0&0\\
				\frac{s-9}{5}&-1&1&0&0\\
				\frac{s-9}{5}&0&-1&1&0\\
				\frac{s-9}{5}&0&0&-1&1\\
				\frac{2s-3}{5}&0&0&0&-1 
			\end{bmatrix}, \rf(s-5)= \begin{bmatrix}
				-1&0&1&0&0\\
				\frac{s-9}{5}&-1&0&1&0\\
				\frac{s-9}{5}&0&-1&0&1\\
				\frac{2s-3}{5}&0&0&-1&0\\
				\frac{s+1}{5}&1&0&0&-1 
			\end{bmatrix}, \] 
			\[ \rf(s-3)= \begin{bmatrix}
				-1&0&0&1&0\\
				\frac{s-9}{5}&-1&0&0&1\\
				\frac{2s-3}{5}&0&-1&0&0\\
				\frac{s+1}{5}&1&0&-1&0\\
				\frac{s+1}{5}&0&1&0&-1 
			\end{bmatrix} \text{or} \begin{bmatrix}
				-1&0&0&1&0\\
				\frac{s-9}{5}&-1&0&0&1\\
				\frac{2s-3}{5}&0&-1&0&0\\
				\frac{s+1}{5}&1&0&-1&0\\
				1&2&0&0&-1 
			\end{bmatrix},\]
			\[ \rf(s-1)= \begin{bmatrix}
				-1&0&0&0&1\\
				\frac{2s-3}{5}&-1&0&0&0\\
				\frac{s+1}{5}&1&-1&0&0\\
				\frac{s+1}{5}&0&1&-1&0\\
				\frac{s+1}{5}&0&0&1&-1 
			\end{bmatrix} \text{or}
			\begin{bmatrix}
				-1&0&0&0&1\\
				\frac{2s-3}{5}&-1&0&0&0\\
				\frac{s+1}{5}&1&-1&0&0\\
				\frac{s+1}{5}&0&1&-1&0\\
				1&1&1&0&-1 
			\end{bmatrix} \]
			\[ \text{or} \begin{bmatrix}
				-1&0&0&0&1\\
				\frac{2s-3}{5}&-1&0&0&0\\
				\frac{s+1}{5}&1&-1&0&0\\
				1&2&0&-1&0\\
				\frac{s+1}{5}&0&0&1&-1 
			\end{bmatrix} \text{or}
			\begin{bmatrix}
				-1&0&0&0&1\\
				\frac{2s-3}{5}&-1&0&0&0\\
				\frac{s+1}{5}&1&-1&0&0\\
				1&2&0&-1&0\\
				1&1&1&0&-1 
			\end{bmatrix}. \]
		\end{proposition}
		
		
		\begin{proposition} \label{55}
			Let $S$ be an Arf numerical semigroup with multiplicity $5$ and conductor $s>5$. If $s \equiv 4\mod 5$ and $S=\langle 5,s,s+2,s+3,s+4\rangle$, then $\pf(S)=\{ s-5,s-3,s-2,s-1 \}$ and
			
						
						\[ \rf(s-5)= \begin{bmatrix}
							-1&1&0&0&0\\
							\frac{s-9}{5}&-1&0&0&1\\
							\frac{2s-3}{5}&0&-1&0&0\\
							\frac{s-4}{5}&0&1&-1&0\\
							\frac{s-4}{5}&0&0&1&-1 
						\end{bmatrix}, \] 
						\[ \rf(s-3)= \begin{bmatrix}
							-1&0&1&0&0\\
							\frac{2s-3}{5}&-1&0&0&0\\
							\frac{s-4}{5}&0&-1&1&0\\
							\frac{s-4}{5}&0&0&-1&1\\
							\frac{s+1}{5}&1&0&0&-1 
						\end{bmatrix} \text{or} \begin{bmatrix}
							-1&0&1&0&0\\
							\frac{2s-3}{5}&-1&0&0&0\\
							\frac{s-4}{5}&0&-1&1&0\\
							0&2&0&-1&0\\
							\frac{s+1}{5}&1&0&0&-1 
						\end{bmatrix}, \]
						\[ \rf(s-2)= \begin{bmatrix}
							-1&0&0&1&0\\
							\frac{s-4}{5}&-1&1&0&0\\
							\frac{s-4}{5}&0&-1&0&0\\
							\frac{s+1}{5}&1&0&-1&0\\
							\frac{2s+2}{5}&0&0&0&-1 
						\end{bmatrix} \text{or} \begin{bmatrix}
							-1&0&0&1&0\\
							\frac{s-4}{5}&-1&1&0&0\\
							\frac{s-4}{5}&0&-1&0&0\\
							\frac{s+1}{5}&1&0&-1&0\\
							0&1&1&0&-1 
						\end{bmatrix} \]
						\[ \text{or} \begin{bmatrix}
							-1&0&0&1&0\\
							\frac{s-4}{5}&-1&1&0&1\\
							0&2&-1&0&1\\
							\frac{s+1}{5}&1&0&-1&0\\
							\frac{2s+2}{5}&0&0&0&-1 
						\end{bmatrix} \text{or} \begin{bmatrix}
							-1&0&0&1&0\\
							\frac{s-4}{5}&-1&1&0&1\\
							0&2&-1&0&1\\
							\frac{s+1}{5}&1&0&-1&0\\
							0&1&1&0&-1 
						\end{bmatrix}, \]
						\[ \rf(s-1)= \begin{bmatrix}
							-1&0&0&0&1\\
							\frac{s-4}{5}&-1&0&1&0\\
							\frac{s+1}{5}&1&-1&0&0\\
							\frac{2s+2}{5}&0&0&-1&0\\
							\frac{s+1}{5}&0&1&0&-1 
						\end{bmatrix} \text{or} \begin{bmatrix}
							-1&0&0&0&1\\
							\frac{s-4}{5}&-1&0&1&0\\
							\frac{s+1}{5}&1&-1&0&0\\
							\frac{2s+2}{5}&0&0&-1&0\\
							0&1&0&1&-1 
						\end{bmatrix} \]
						\[ \text{or} \begin{bmatrix}
							-1&0&0&0&1\\
							\frac{s-4}{5}&-1&0&1&0\\
							\frac{s+1}{5}&1&-1&0&0\\
							0&1&1&-1&0\\
							\frac{s+1}{5}&0&1&0&-1 
						\end{bmatrix} \text{or} \begin{bmatrix}
							-1&0&0&0&1\\
							\frac{s-4}{5}&-1&0&1&0\\
							\frac{s+1}{5}&1&-1&0&0\\
							0&1&1&-1&0\\
							0&1&0&1&-1 
						\end{bmatrix}. \]
					\end{proposition}
					
					One can calculate the determinant of all row-factorization matrices of $s-1$ in propositions given in this section by determinant laws or by using any software available online, and check there is always a row-factorization matrix of $s-1$ whose determinant's absolute value is $s-1$ itself. 
					
					\begin{corollary} \label{cor1}
						Let $S$ be an Arf numerical semigroup with multiplicity up to five and conductor $s$. Then there exist a row-factorization matrix $\rf(s-1)$ of the Frobenius number $\fr(S)=s-1$ such that $$|\det\rf(s-1)|=s-1.$$
					\end{corollary}
					
					\begin{exam}
						Let $S$ be the Arf numerical semigroup $\langle 5,19,21,22,23 \rangle$. Then $\co(S)=19$, $\pf(S)=\{ 14,16,17,18 \}$ and
						\[ \rf(18)= \begin{bmatrix}
							-1&0&0&0&1\\
							3&-1&0&1&0\\
							4&1&-1&0&0\\
							8&0&0&-1&0\\
							4&0&1&0&-1 
						\end{bmatrix} \text{or} \begin{bmatrix}
							-1&0&0&0&1\\
							3&-1&0&1&0\\
							4&1&-1&0&0\\
							8&0&0&-1&0\\
							0&1&0&1&-1 
						\end{bmatrix} \]
						\[ \text{or} \begin{bmatrix}
							-1&0&0&0&1\\
							3&-1&0&1&0\\
							4&1&-1&0&0\\
							0&1&1&-1&0\\
							4&0&1&0&-1 
						\end{bmatrix} \text{or} \begin{bmatrix}
							-1&0&0&0&1\\
							3&-1&0&1&0\\
							4&1&-1&0&0\\
							0&1&1&-1&0\\
							0&1&0&1&-1 
						\end{bmatrix}, \]
						\[
						\text{and } \det \begin{bmatrix}
							-1&0&0&0&1\\
							3&-1&0&1&0\\
							4&1&-1&0&0\\
							8&0&0&-1&0\\
							4&0&1&0&-1 
						\end{bmatrix}=18.
						\]
					\end{exam}
					
					Note that the parametrizations of Arf numerical semigroups with multiplicity six are given in \cite[Section 7.6]{ghkr}. There are $12$ different characterizations of such numerical semigroups and each of them has type $5$. This means that if we want to find row-factorization matrices of such numerical semigroups, we will get at least $60$ different matrices which can be calculated in a similar way. 
					
					\section{Row-Factorization Matrices of Arf numerical semigroups with multiplicity bigger than 5}\label{sec4}
					
					In this section, we investigate the row-factorization matrices of Arf numerical semigroups whose conductor is a multiple of the multiplicity.
					
					\begin{lemma} \label{arflemma}
						Let $S$ be numerical semigroup with multiplicity $\m(S)=m$ and conductor $\co(S)=s$ such that $s \equiv 0 \mod m$. If $S$ is minimally generated by the set $\{ m,s+1,s+2,\dots,s+m-1 \}$, then $S$ is Arf.
					\end{lemma}
					
					\begin{proof}
						Let $x,y\in S$ be with $x\geq y$. If $x\geq s$, then $2x-y\geq x\geq s$, and so $2x-y\in S$. If $x<s$, then there exist $k_{1},k_{2}\in \N$ such that $k_{1}\geq k_{2}$ and $x=mk_{1}, y=mk_{2}$. Then $2x-y=2mk_{1}-mk_{2}=m(2k_{1}-k_{2}) \in S$. Hence, $S$ is an Arf numerical semigroup.
					\end{proof}
					
					\begin{proposition} \label{p10}
						Let $S$ be an Arf numerical semigroup of the form in which $S=\langle m,s+1,s+2,\dots,s+m-1 \rangle$ with multiplicity $m$, conductor $s$ and $s \equiv 0 \mod m$. Then $\pf(S)=\{ s-m+1,s-m+2,\dots,s-1 \}$ and for each $k\in\{1,\dots,m-1\}$, $s-k\in\pf(S)$ has a row-factorization matrix $\rf(s-k)=(a_{ij})_{m\times m}$ where
						\[
						a_{ij}=
						\begin{cases}
							-1 & i=j\\
							1 & i=1, j=m-k+1\\
							\frac{2s}{m} & i=k+1, j=1\\
							\frac{s-m}{m} & 1<i<k+1, j=1\\
							1 & 1<i<k+1, j=i+m-k\\
							\frac{s}{m} & i>k+1, j=1\\
							1 & i>k+1, j=i-k\\
							0 & \text{otherwise}
						\end{cases}.
						\]
					\end{proposition}
					
					\begin{proof}
						Since $S=\langle m,s+1,s+2,\dots,s+m-1 \rangle$ is Arf, we have $\ap(S,m)=\{ 0,s+1,s+2,\dots,s+m-1 \}$. Therefore, $\pf(S)=\{ s-m+1,s-m+2,\dots,s-1 \}$. Take $s-k\in\pf(S)$ for some $k\in\{1,\dots,m-1\}$. Let $\rf(s-k)=(a_{ij})_{m\times m}$ be one of its row-factorization matrices. By definition we know that $a_{ij}=-1$ when $i=j$. Thus, $s-k$ has fallowing factorizations
						\[\begin{array}{rlllllrll}
							s-k&=&-1\times m&+&a_{12}\times(s+1)&+&\dots&+&a_{1m}\times(s+m-1),\\
							s-k&=&a_{21}\times m&+&-1\times(s+1)&+&\dots&+&a_{2m}\times(s+m-1),\\
							\vdots& &\vdots& &\vdots& & & &\vdots\\
							s-k&=&a_{m1}\times m&+&a_{m2}\times(s+1)&+&\dots&+&-1\times(s+m-1).\\
						\end{array}\]
						We now investigate the values of coefficients $a_{ij}$ in four cases.
						
						Case 1: When $i=1$, we have
						\[
						s-k=-1\times m+a_{12}\times(s+1)+\dots+a_{1m}\times(s+m-1).
						\]
						Then
						\[
						s+m-k=a_{12}\times(s+1)+\dots+a_{1m}\times(s+m-1).
						\]
						By definition of row-factorization matrices, for each $j\neq 1$ we have $a_{1j}\in\N_{0}$. Since $s+m-k$ is one of the minimal generators of $S$, it can not factor in a non-trivial way by the minimality. Then we get $s+m-k=(s+j-1)$, and so $j=m-k+1$.
						
						Case 2: When $i=k+1$, we have
						\[
						s-k=a_{(k+1)1}\times m+a_{(k+1)2}\times(s+1)+\dots+(-1)\times(s+k)+\dots+a_{(k+1)m}\times(s+m-1).
						\]
						Then
						\[
						2s=a_{(k+1)1}\times m+a_{(k+1)2}\times(s+1)+\dots+\check{\overbrace{(-1)\times(s+k)}}+\dots+a_{(k+1)m}\times(s+m-1),
						\]
						where $\check{\overbrace{(-)}}$ means that the summand is taken out from the equation. By definition of row-factorization matrices, for each $j\neq k+1$ we have $a_{(k+1)j}\in\N_{0}$. Therefore, we get a factorization when $a_{(k+1)1}=\frac{2s}{m}$ and $a_{(k+1)j}=0$ for the other values of $j$.
						
						Case 3: When $1<i<k+1$, we have
						\[
						s-k=a_{i1}\times m+a_{i2}\times(s+1)+\dots+(-1)\times(s+i-1)+\dots+a_{im}\times(s+m-1).
						\]
						Then
						\[
						2s-k+i-1=a_{i1}\times m+a_{i2}\times(s+1)+\dots+\check{\overbrace{(-1)\times(s+i-1)}}+\dots+a_{im}\times(s+m-1).
						\]
						Similarly, for each $j\neq i$ we have $a_{ij}\in\N_{0}$ and so, we get one of the possible factorizations when $a_{i1}=\frac{s-m}{m}$, $a_{ij}=1$ for some $j\neq 1$ and $a_{ij}=0$ for the other values of $j$. Then we get $2s-k+i-1=\frac{s-m}{m}m+(s+j-1)$. This means that $a_{ij}=1$ whenever $j=m+i-k$.
						
						Case 4: When $i>k+1$, we similarly get a factorization when $a_{i1}=\frac{s}{m}$, $a_{ij}=1$ for some $j\neq 1$ and $a_{ij}=0$ for the other values of $j$. Then we get $2s-k+i-1=\frac{s}{m}m+(s+j-1)$. This means that $a_{ij}=1$ whenever $j=i-k$.
						
					\end{proof}
					
					\begin{corollary} \label{cor2}
						Let the notation and situation as in Proposition \ref{p10}. Then we have $|\det \rf(s-1)|=s-1$.
					\end{corollary}
					
					\begin{proof}
						By Proposition \ref{p10}, we find one of the RF-matrices of $s-1$ as follows, which is an $m\times m$ matrix,
						\[
						\rf(s-1)=\begin{bmatrix}
							-1&0&\dots&\dots&0&1\\
							\frac{2s}{m}&-1&0&\dots&\dots&0\\
							\frac{s}{m}&1&-1&0&\dots&0\\
							\vdots&0&\ddots&\ddots&\ddots&\vdots\\
							\vdots&\vdots&\ddots&\ddots&-1&0\\
							\frac{s}{m}&0&\dots&0&1&-1 
						\end{bmatrix}.
						\]
						As we know that adding a multiple of one row to another does not change the determinant, to find the determinant of $\rf(s-1)$, we will do the following. We multiply the first row by $\frac{2s}{m}$ and add to the second row, and multiply first row by $\frac{s}{m}$ and add to the other rows to get
						\[
						\begin{bmatrix}
							-1&0&\dots&\dots&0&1\\
							0&-1&0&\dots&\dots&\frac{2s}{m}\\
							0&1&-1&0&\dots&\frac{s}{m}\\
							\vdots&0&\ddots&\ddots&\ddots&\vdots\\
							\vdots&\vdots&\ddots&\ddots&-1&\frac{s}{m}\\
							0&0&\dots&0&1&\frac{s}{m}-1 
						\end{bmatrix}.
						\]
						Then starting with $i=2$, we add $i$th row to $(i+1)$th row in succession to get
						\[
						\begin{bmatrix}
							-1&0&\dots&\dots&0&1\\
							0&-1&0&\dots&\dots&\frac{2s}{m}\\
							0&0&-1&0&\dots&\frac{3s}{m}\\
							\vdots&0&\ddots&\ddots&\ddots&\vdots\\
							\vdots&\vdots&\ddots&\ddots&-1&\frac{(m-1)s}{m}\\
							0&0&\dots&0&0&\frac{ms}{m}-1 
						\end{bmatrix}.
						\]
						Then we get $\det\rf(s-1)=(-1)^{m-1}(s-1)$. Therefore, we find $|\det \rf(s-1)|=s-1$ as desired.
					\end{proof}
					
					
					\begin{remark} \label{r}
						Let $S$ be an Arf numerical semigroup with multiplicity $m$ bigger than $5$ and conductor $s$. We know that the set $$(\ap(S,m)\cup\{m\})\setminus\{0\}=\{ m,\omega(1),\dots,\omega(m-1) \}$$ is the minimal generating set for $S$. By \cite[Lemma 11]{ghkr}, for each $j\in\{ 2,\dots,m-1 \}$, we have $\omega(j-1)\geq s$ or $\omega(j)\geq s$. Also by \cite[Lemma 13]{ghkr}, we have $\omega(1)=s+1$ or $s-\bar{s}+m+1$ and $\omega(m-1)=s-\bar{s}+m-1$ where $\bar{s}$ is the reminder of $s$ divided by $m$. This shows that such Arf numerical semigroups has at least $3$ generators greater than or equal to its conductor $s$.
					\end{remark}
					
					\begin{lemma} \label{lm}
						Let $S$ be an Arf numerical semigroup minimally generated by $\{ n_{1}<\dots<n_{m}\}$ with multiplicity $m$ bigger than $5$ and conductor $s$. If $\rf(s-1)=(a_{ij})$ is a row-factorization matrix of the Frobenius number, then $a_{ij}=a_{i'j}=0$ for some $i,i',j\in\{1,\dots,m\}$.
					\end{lemma}
					
					\begin{proof}
						By Remark \ref{r}, $S$ has at least $3$ generators greater than or equal to $s$. So, let $n_{i}<n_{k}<n_{l}$ be three generators of $S$ such that $s\leq n_i$. Let $n_{i'}$ be another generator of $S$ different from $n_{i}$, $n_{k}$ and $n_{l}$. If $n_{i'}<n_{i}$, then we can find the first and $i'$th rows of $\rf(s-1)$ using the following factorizations of $s-1$.
						\[\begin{array}{ll}
							s-1&=(-1)\times m+\dots+a_{1i'}\times n_{i'}+\dots +a_{1i}\times n_{i}+\dots+a_{1k}\times n_{k}+\dots\\ &\dots+a_{1l}\times n_{l}+\dots+a_{1m}\times n_{m},\\
							s-1&=a_{i'1}\times m+\dots +(-1)\times n_{i'}+\dots+a_{i'i}\times n_{i}+\dots+a_{i'k}\times n_{k}+\dots\\ &\dots+a_{i'l}\times n_{l}+\dots+a_{i'm}\times n_{m}.\\
						\end{array}\]
						Since $s+m-1<2s$, at most one of $a_{1i},a_{1k},a_{1l}$ can be positive. Similarly, since $s+n_{i'}-1<n_{i}+n_{k}$, at most one of $a_{i'i},a_{i'k},a_{i'l}$ can be positive. This means that $a_{1j}=a_{i'j}=0$ for some $j\in\{ i,k,l \}$.
						
						If $n_{i}<n_{i'}$, without loss of generality we can assume that $n_{i}<n_{i'}<n_{k}<n_{l}$. Then we can find the first and $i$th rows of $\rf(s-1)$ using the following factorizations of $s-1$.
						\[\begin{array}{ll}
							s-1&=(-1)\times m+\dots +a_{1i}\times n_{i}+\dots+a_{1i'}\times n_{i'}+\dots +a_{1k}\times n_{k}+\cdots \\
							&\dots+a_{1l}\times n_{l}+\dots+a_{1m}\times n_{m},\\
							s-1&=a_{i1}\times m+\dots +(-1)\times n_{i}+\dots+a_{ii'}\times n_{i'}+\dots +a_{ik}\times n_{k}+\cdots \\ &\dots+a_{il}\times n_{l}+\dots+a_{im}\times n_{m}.\\
						\end{array}\]
						Since $s+m-1<2s$, at most one of $a_{1i'},a_{1k},a_{1l}$ can be positive. Similarly, since $s+n_{i}-1<n_{i'}+n_{k}$, at most one of $a_{ii'},a_{ik},a_{il}$ can be positive. This means that $a_{1j}=a_{ij}=0$ for some $j\in\{ i',k,l\}$.
					\end{proof}

					\section{Defining Ideals of Arf Rings}\label{sec5}
					
					In this section, we investigate the defining ideals of numerical semigroup rings obtained by Arf numerical semigroups.
					
					Let $S$ be a numerical semigroup minimally generated by the set $\{ n_{1},n_{2},\dots,n_{e} \}$, $\Bbbk$ an algebraically closed field with zero characteristics and $t$ an indeterminate. The semigroup ring $\Bbbk[S]=\Bbbk[t^{s} \mid s\in S]$ of $S$ is a subring of the one dimensional polynomial ring $\Bbbk[t]$. The semigroup ring $\Bbbk[\![S]\!]=\Bbbk[\![t^{s} \mid s\in S]\!] $ associated to $S$ is a subring of the one dimensional power series ring $\Bbbk[\![t]\!]$. Notice that $\Bbbk[\![S]\!]$ is $\mathfrak{n}$-adic completion of the local ring $\Bbbk[S]_{\mathfrak{m}}$ where $\mathfrak{m}=\langle t^{n_{1}},\dots,t^{n_{e}} \rangle$ and $\mathfrak{n}=\mathfrak{m}\Bbbk[S]_{\mathfrak{m}}$.
					
					For a vector $\bold{a}=(a_{i})\in \Z^{e}$, we define the support of $\bold{a}$ as $\supp \bold{a}=\{ i \mid a_{i}\neq 0 \}$, and the degree map with respect to $S$ as $\deg_{S}: \Z^{e}\rightarrow \Z$ by $\deg_{S}\bold{a}=\sum_{i=1}^{e}a_{i}n_{i}$. Since $S$ is numerical the degree map is surjective and $\kerr \deg_{S}:= V(S)$ is a torsion-free $\Z$-submodule of $\Z^{e}$ of rank $e-1$.
					
					We know that there is a surjective ring homomorphism from the polynomial ring $\Bbbk[x_{1},\dots,x_{e}]$ to $\Bbbk[S]$ sending $x_{i}$ to $t^{n_{i}}$ whose kernel is called the defining ideal of $\Bbbk[S]$ and denoted by $I_{S}$. Let $\bold{x}^{\bold{a}}$ denote the monomial $x_{1}^{a_{1}}x_{2}^{a_{2}}\cdots x_{e}^{a_{e}}$ for a vector $\bold{a}\in\Z^{e}$ whose components are non-negative. Then $I_{S}$ is a binomial ideal generated by the set $\{ \bold{x}^{\bold{a}}-\bold{x}^{\bold{b}} \mid \deg_{S}\bold{a}=\deg_{S}\bold{b} \}$. Since $\Bbbk[S]$ is an integral domain, $I_{S}$ is a prime ideal, and so $I_{S}$ is a toric ideal. 
					
					For a binomial $\phi=\bold{x}^{\bold{a}}-\bold{x}^{\bold{b}}$, we write $\supp \phi= \supp\bold{a} \cup \supp\bold{b}$, and we say that $\phi$ has full support if $\supp \phi=\{ 1,2,\dots,e \}$. If a toric ideal has a minimal generating system consisting of binomials with full support, then we call it generic. In \cite{eto1}, K. Eto characterized the generecity of $I_{S}$ in terms of the row-factorization matrices of $S$.
					
					\begin{theorem} \label{eto1} \cite[Theorem 1]{eto1}
						$I_{S}$ is generic if and only if $\rf(f)=(a_{ij})$ is unique and $a_{ij}\neq a_{i'j}$ when $i\neq i'$ for each $f\in\pf(S)$.
					\end{theorem}
					
					For a vector $\bold{a}=(a_{i})\in \Z^{e}$, we write $\bold{a}^{+}$ for the vector whose $i$th component is $a_{i}$ if $a_{i}> 0$, and is zero if $a_{i}\leq0$, and we write $\bold{a}^{-}$ for the vector $\bold{a}^{+}-\bold{a}$. Then $\bold{a}=\bold{a}^{+}-\bold{a}^{-}$ and $\bold{a}^{+},\bold{a}^{-}\in\N_{0}^{e}$. Notice that the generating set $\{ \bold{x}^{\bold{a}}-\bold{x}^{\bold{b}} \mid \deg_{S}\bold{a}=\deg_{S}\bold{b} \}$ for $I(S)$ is equal to the set $\{ \bold{x}^{\bold{v}^{+}}-\bold{x}^{\bold{v}^{-}} \mid \bold{v}\in V(S) \}$. It is known that there are $\bold{v}_1,\dots,\bold{v}_k\in V(S)$ such that $I(S)$ is generated by $\{ \bold{x}^{\bold{v}^{+}}-\bold{x}^{\bold{v}^{-}} \mid \bold{v}=\bold{v}_i \text{ for some } 1\leq i\leq k \}$. However, we can not choose $\bold{v}_i$'s from $V(S)$ in general. 
					
					Now let $\rf(\fr)=(a_{ij})$ be a row-factorization matrix of the Frobenius number $\fr(S)$ of $S$. Let $\bold{a}_{1},\dots,\bold{a}_{e}$ be the row vectors of $\rf(\fr)$ and set $\bold{a}_{ij}=\bold{a}_{i}-\bold{a}_{j}$ for all $1\leq i<j\leq e$. Then $\bold{a}_{ij}\in V(S)$ and $\sum_{i<j}\bold{a}_{ij}\Z:=W(S)$ is a $\Z$-submodule of $V(S)$. In some cases, we can choose $\bold{v}_i$'s as exactly $\bold{a}_{ij}$'s, and we get that $I_{S}$ is generated by the binomials $\phi_{ij}=\bold{x}^{\bold{a_{ij}}^{+}} - \bold{x}^{\bold{a_{ij}}^{-}}$, such a binomial is called $\rf(F)$-relation. For example, when $S$ is a numerical semigroup with multiplicity $3$ or $S$ is an almost symmetric numerical semigroup with multiplicity $4$, it is shown by J. Herzog and K. Watanabe in \cite{HW} that $I_{S}$ is generated by $\rf(F)$-relations.
					
					In \cite[Note 1.1, Note 5.1]{eto3} and in \cite[Note 1.2]{eto2}, K. Eto states that $V(S)=W(S)$ if and only if there exists a row-factorization matrix of $\fr(S)$ such that the absolute value of its determinant is $\fr(S)$ itself. This makes the problem of choosing $\bold{v}_i$'s easier. Then we can observe the next theorem as a result.
					
					\begin{theorem} \label{eto2}
						The defining ideal $I_{S}$ is generated by the set $\{ \bold{x}^{\bold{v}^{+}}-\bold{x}^{\bold{v}^{-}} \mid \bold{v}\in W(S) \}$ if and only if there is a row-factorization matrix of $\fr(S)$ with $|\det \rf(F)|=\fr(S)$.
					\end{theorem}
					
					Eto also gives a conjecture on determinants of row-factorization matrices of the Frobenius number $\fr(S)$ of $S$ in \cite{eto2}.
					
					\begin{conjecture} \label{conj}\cite[Conjecture 1.1]{eto2}
						For a numerical semigroup $S$ with embedding dimension $\ed(S)=e$, there is a row-factorization matrix $\rf(\fr)$ for $\fr(S)$ such that $\det \rf(\fr)=(-1)^{e+1}\fr(S)$.
					\end{conjecture}
					
					Eto proves his conjecture for some particular cases, for example, when $S$ is symmetric or almost symmetric with embedding dimension $4$, see \cite[Theorem 1.2]{eto2}. Next we prove the conjecture is satisfied for some Arf numerical semigroups too.
					
					\begin{theorem} \label{thm1}
						Conjecture \ref{conj} is true for
						\begin{enumerate}
							\item Arf numerical semigroups with multiplicity up to five,
							\item Arf numerical semigroups whose conductor is a multiple of the multiplicity.
						\end{enumerate}
					\end{theorem}
					
					\begin{proof}
						The proof of (1) is due to Corollary \ref{cor1}, and the proof of (2) is due to Corollary \ref{cor2}.
					\end{proof}
					
					\begin{theorem} \label{thm2}
						Let $S$ be an Arf numerical semigroup with multiplicity up to five or whose conductor is a multiple of the multiplicity. Then the defining ideal $I_{S}$ of $\Bbbk[S]$ is generated by the set $\{ \bold{x}^{\bold{v}^{+}}-\bold{x}^{\bold{v}^{-}} \mid \bold{v}\in W(S) \}$.
					\end{theorem}
					
					\begin{proof}
						Use Theorems \ref{thm1} and \ref{eto2}.
					\end{proof}
					
					\begin{exam}
						Let $S$ be the Arf numerical semigroup minimally generated by the set $\{ 4,10,21,23 \}$ with $k=2$ and $\co(S)=20$ as in Proposition \ref{p3}. Then $\fr(S)=19$, and one of row-factorization matrices of $\fr(S)$ is
						\[
						\rf(19)= \begin{bmatrix}
							-1&0&0&1\\
							2&-1&1&0\\
							10&0&-1&0\\
							8&1&0&-1 \end{bmatrix}
						\]
						where $\det\rf(19)=-19$. Therefore, by Theorem \ref{thm2}, the defining ideal $I_{S}$ of $\Bbbk[S]=\Bbbk[t^{4},t^{10},t^{21},t^{23}]$ is generated by the set $\{ \bold{x}^{\bold{v}^{+}}-\bold{x}^{\bold{v}^{-}} \mid \bold{v}\in W(S) \}$. Since the row vectors of $\rf(19)$ are $\bold{a}_{1}=(-1,0,0,1)$, $\bold{a}_{2}=(2,-1,1,0)$, $\bold{a}_{3}=(10,0,-1,0)$ and $\bold{a}_{4}=(8,1,0,-1)$, we find the generators of $W(S)$ as follows
						\[
						\begin{array}{lll}
							\bold{a}_{12}=(-3,1,-1,1), & & \bold{a}_{13}=(-11,0,1,1),\\
							\bold{a}_{14}=(-9,-1,0,2), & & \bold{a}_{23}=(-8,-1,2,0),\\
							\bold{a}_{24}=(-6,-2,1,1), & & \bold{a}_{34}=(2,-1,-1,1).\\
						\end{array}
						\]
						Then $W(S)=\sum_{i=1}^{3} \sum_{j=i+1}^{4} \bold{a}_{ij}\Z$, where $\bold{a}_{ij}\Z$ is the $\Z$-submodule of $\Z^4$ generated by $\bold{a}_{ij}$, and
						
						$$I_{S}=\langle \{ \bold{x}^{\bold{v}^{+}}-\bold{x}^{\bold{v}^{-}} \mid \bold{v}\in \sum_{i=1}^{3} \sum_{j=i+1}^{4} \bold{a}_{ij}\Z \} \rangle.$$
						
						However, one can find $\rf(\fr)$-relations as follows
						\[
						\begin{array}{lll}
							\phi_{12}=x_{2}x_{4}-x_{1}^{3}x_{3} & & \phi_{13}=x_{3}x_{4}-x_{1}^{11}\\
							\phi_{14}=x_{4}^{2}-x_{1}^{9}x_{3} & & \phi_{23}=x_{3}^{2}-x_{1}^{8}x_{2}\\
							\phi_{24}=x_{3}x_{4}-x_{1}^{6}x_{2}^{2} & & \phi_{12}=x_{1}^{2}x_{4}-x_{2}x_{3},
						\end{array}
						\]
						and use Macaulay2 from \cite{M2} to see that $I_{S}$ is indeed minimally generated by $\rf(\fr)$-relations. This means that $$I_{S}=\langle x_{2}x_{4}-x_{1}^{3}x_{3},x_{3}x_{4}-x_{1}^{11},x_{4}^{2}-x_{1}^{9}x_{3},x_{3}^{2}-x_{1}^{8}x_{2}, x_{3}x_{4}-x_{1}^{6}x_{2}^{2},x_{1}^{2}x_{4}-x_{2}x_{3} \rangle$$ and $\Bbbk[S]\cong \Bbbk[x_{1},x_{2},x_{3},x_{4}]/I_{S}$.
					\end{exam}
					
					\begin{theorem} \label{thm3}
						Let $S$ be an Arf numerical semigroup with multiplicity smaller than four. Then the defining ideal $I_{S}$ of $\Bbbk[S]$ is generic.
					\end{theorem}
					
					\begin{proof}
						Combine Propositions \ref{p1} and \ref{p2} with Theorem \ref{eto1}.
					\end{proof}
					
					\begin{theorem} \label{thm4}
						Let $S$ be an Arf numerical semigroup with multiplicity bigger than $3$. Then the defining ideal $I_{S}$ of $\Bbbk[S]$ is not generic.
					\end{theorem}
					
					\begin{proof}
						Combine Propositions \ref{p3}, \ref{p31}, \ref{p4}, \ref{51}, \ref{52}, \ref{53}, \ref{54}, \ref{55}, \ref{p10} and Lemma \ref{lm} with Theorem \ref{eto1}.
					\end{proof}
					
					From now on we will focus on Arf rings which were originally studied in the classification of singularities of plane curves by C. Arf in \cite{arf}. Then J. Lipman in \cite{lip} introduced the notions of Arf rings and Arf closure of rings by extracting the essence of the rings considered by C. Arf. 
					\begin{definition}
						Let $R$ be a Noetherian semi-local ring such that $R_{\mathfrak{m}}$ is a one-dimensional Cohen-Macaulay local ring for every maximal ideal $\mathfrak{m}$ of $R$ and $Q(R)$ be the total ring of fractions of $R$. Then $R$ is called an Arf ring the the following conditions hold:
						\begin{enumerate}
							\item For every integrally closed ideal $I$ of $R$, $I^{n+1}=aI^{n}$ for some $a\in I$ and $n\geq 0$.
							\item If $x,y,z \in R$ with $x$ is a non-zero divisor element of $R$ and $\frac{y}{x},\frac{z}{x}\in Q(R)$ are integral over $R$, then $\frac{yz}{x}\in R$.
						\end{enumerate}
					\end{definition}
					
					Since Arf property is preserved under localization and completion, we can think of $R$ as to be a one-dimensional complete local Noetherian domain containing an algebraically closed field of zero characteristic. Among all the Arf rings between $R$ and the integral closure of $R$, there is a smallest one which is called the Arf closure of $R$. The value semigroups of Arf closures appear as Arf numerical semigroups. When $R$ is an Arf ring of the form $\Bbbk[\![S]\!]$ where $S$ is an Arf numerical semigroup, using Theorems \ref{thm2}, \ref{thm3} and \ref{thm4} one can easily conclude following corollaries.
					
					
					
					
					\begin{corollary} \label{corro}
						Let $R$ be an Arf ring of the form $\Bbbk[\![S]\!]$ where $S$ is an Arf numerical semigroup with multiplicity up to five or whose conductor is a multiple of the multiplicity. Then the defining ideal of $R$ is the ideal generated by $\{ \bold{x}^{\bold{v}^{+}}-\bold{x}^{\bold{v}^{-}} \mid \bold{v}\in W(S) \}$.
					\end{corollary}
					
					\begin{corollary} \label{corroo}
						If $R$ is an Arf ring of the form $\Bbbk[\![S]\!]$ where $S$ is an Arf numerical semigroup with multiplicity  smaller than four, then the defining ideal of $R$ is generic, and it is not generic otherwise.  
					\end{corollary}
					

					\section*{Acknowledgements} We thank Naoki Endo and Om Prakash Bhardwaj for reading the first drafts and their valuable comments. We also thank the anonymous reviewers for their careful reading and their many insightful comments and suggestions.

				\end{document}